\documentclass[12pt]{amsart}
\usepackage{amssymb, amstext, amscd, amsmath}
\usepackage{fullpage}

\theoremstyle{plain}
\newtheorem{thm}{Theorem}[section]

\newtheorem{lem}[thm]{Lemma}

\theoremstyle{definition}
\newtheorem{defn}[thm]{Definition}

\newtheorem{ques}[thm]{Question}

\theoremstyle{remark}
\newtheorem{rem}[thm]{Remark}

\begin{document}

\title{An example of weakly amenable and character amenable operator}
\title{An example of weakly amenable and character amenable operator}
\author[L. Y. Shi]{Luo Yi Shi}

\address{Department of Mathematics\\Tianjin Polytechnic University\\Tianjin 300160\\P.R. CHINA}

\email{sluoyi@yahoo.cn}
\author[Y. J. Wu]{YU Jing Wu}

\address{Tianjin Vocational Institute \\Tianjin 300160\\P.R. CHINA}

\email{wuyujing111@yahoo.cn}
\author[Y.Q. Ji]{You Qing Ji}

\address{Department of Mathematics\\Jilin University\\Changchun 130012\\P.R. CHINA}

\email{jiyq@jlu.edu.cn}

\thanks{Supported by NCET(040296), NNSF of China(10971079) and
        the Specialized Research Fund for the Doctoral Program
        of Higher Education(20050183002)}

\date{June 15 2010}



\begin{abstract}
A complete characterization of Hilbert space operators that generate weakly amenable algebras remains open, even
in the case of compact operator.  Farenick, Forrest and Marcoux proposed the question that if $T$ is a compact
weakly amenable operator on a Hilbert space $\mathfrak{H}$, then is $T$ similar to a normal operator?
 In this
paper we demonstrate an example of compact triangular operator on infinite-dimension Hilbert space which is a
weakly amenable and character amenable operator but is not similar to a normal operator.

\end{abstract}

\maketitle

\section{Introduction}

Let $\mathfrak{A}$ be a Banach algebra, and let $X$ be a Banach
$\mathfrak{A}$-bimodule. A {\it derivation} $D :
\mathfrak{A}\rightarrow X$ is a continuous linear map such that
$D(ab) = a \cdot D(b) + D(a)\cdot b$, for all $a, b \in
\mathfrak{A}$. A  derivation $D : \mathfrak{A}\rightarrow X$ is said
to be {\it inner} if there exists $x\in X$ such that $D(a) = a\cdot
x- x \cdot a$ for all $a\in \mathfrak{A}$. A Banach
$\mathfrak{A}$-bimodule $X$ is said to be {\it commutative} if
$a\cdot x=x\cdot a$ for each $a\in \mathfrak{A}, x\in X$. For any
Banach $\mathfrak{A}$-bimodule $X$, its dual $X^*$ is naturally
equipped with a  Banach $\mathfrak{A}$-bimodule structure via $ [a
\cdot f](x)=f(x \cdot a), [f \cdot a](x)=f(a \cdot x), a\in
\mathfrak{A}, x\in X, f\in X^*.$

We can now give the definition of  amenability, weak amenability and character amenability for Banach algebra:

\begin{defn}
A Banach algebra $\mathfrak{A}$ is {\it amenable} if , for each
Banach $\mathfrak{A}$-bimodules $X$, every derivation  $D :
\mathfrak{A}\rightarrow X^*$ is inner.
\end{defn}

\begin{defn}
A commutative Banach algebra $\mathfrak{A}$ is {\it weakly amenable}
if , for each commutative Banach $\mathfrak{A}$-bimodules $X$, every
derivation  $D : \mathfrak{A}\rightarrow X$ is inner.
\end{defn}

 Let $\mathfrak{A}$ be a Banach algebra
and $\sigma(\mathfrak{A})$  the spectrum of $\mathfrak{A}$, that is, the set of all non-zero multiplicative
linear functionals on $\mathfrak{A}$. If $\varphi \in \sigma(\mathfrak{A})\cup\{0\}$ and
 if $X$ is a Banach space, then $X$ can be viewed as left or
right Banach  $\mathfrak{A}$-module by the following actions. For $a\in \mathfrak{A}, x\in X$:
$$ a\cdot x = \varphi(a)x,\eqno(2.1)$$
$$ x\cdot a =\varphi(a)x.\eqno(2.2)$$
If the left action of $\mathfrak{A}$ on $X$ is given by (2.1), then it is easily verified that the right action
of $\mathfrak{A}$ on the dual $\mathfrak{A}$-module $X^*$ is given by $f\cdot a = \varphi(a) f,$ for all $f\in
X^*, a\in \mathfrak{A}$. Throughout, by  {\it $(\varphi, \mathfrak{A})$-bimodule} $X$, we mean that $X$ is a
Banach $\mathfrak{A}$-bimodule for which the left module action is given by (2.1). {\it $(\mathfrak{A},
\varphi)$-bimodule} is defined similarly by (2.2). Let $\varphi \in \sigma(\mathfrak{A})\cup\{0\}$, a Banach
algebra $\mathfrak{A}$ is said to be {\it left $\varphi$
 amenable}, if every derivation $D$ from $\mathfrak{A}$ into the dual
$A$-bimodule $X^*$ is inner for all  $(\varphi, A)$-bimodules $X$; $\mathfrak{A}$ is said to be {\it right
$\varphi$
 amenable}, if every derivation $D$ from $\mathfrak{A}$ into the dual
$A$-bimodule $X^*$ is inner for all  $(A, \varphi)$-bimodules $X$. $\mathfrak{A}$ is said to be {\it left
character
 amenable}, if it is left $\varphi$ amenable for all
 $\varphi \in \sigma(\mathfrak{A})\cup\{0\}$; $\mathfrak{A}$ is said to be {\it right character
 amenable}, if it is right $\varphi$ amenable for all
 $\varphi \in \sigma(\mathfrak{A})\cup\{0\}$.

\begin{defn}
A  Banach algebra $\mathfrak{A}$ is said to be {\it character
 amenable}, if it is both left character
 amenable and right character amenable.
\end{defn}

The concept of amenable Banach algebras was first introduced by B. E. Johnson in \cite{BE1972}. Weak amenability
was first defined by Bade, Curtis and Dales in \cite{W1987,H2000}. Character amenability was first defined by
Sangani-Monfared \cite{M}. Ever since its introduction, the concepts of amenability, weak amenability and
character amenability have played an important role in research in Banach algebras, operator algebras and
harmonic analysis. We only would like to mention the following deep results due to Willis \cite{W1995} and
Farenick, Forrest and Marcoux \cite{F2005}, \cite{F2007}:

Given a complex separable infinite-dimensional Hilbert space
$\mathfrak{H}$, we write $\mathfrak{B}(\mathfrak{H})$ for the
bounded linear operators on $\mathfrak{H}$. If $T\in
\mathfrak{B}(\mathfrak{H})$, denote the norm-closure of span$\{T^k:
k\in\mathbb{N}\}$ by $\mathfrak{A}_T$, where $\mathbb{N}$ is the set
of natural numbers. $T$ is said to
 be {\it amenable (weakly
amenable or character amenable)} if $\mathfrak{A}_T$ is amenable
(respectively, weakly amenable, character amenable).

In \cite{W1995}, Willis showed that:

\begin{thm}
Suppose $T$ is a compact amenable operator, then $T$ is similar to a
normal operator.
\end{thm}

In \cite{F2005}, \cite{F2007} Farenick, Forrest and Marcoux showed
that:

\begin{thm}
Suppose $T$ is a triangular operator with respect to an orthonormal
basis of $\mathfrak{H}$, then $T$ is amenable if and only if $T$ is
similar to a normal operator  whose spectrum has connected
complement and empty interior.
\end{thm}

A complete characterisation of Hilbert space operators that generate
weakly amenable algebras remains open, even in the case of compact
operator. In \cite{F2005}, \cite{F2007} Farenick, Forrest and
Marcoux proposed the following question:

\begin{ques}
If $T$ is a compact weakly amenable operator on $\mathfrak{H}$, then
is $T$ similar to a normal operator?
\end{ques}

It is well known that if $\mathfrak{H}$ is a finite-dimensional
Hilbert space, then $T\in\mathfrak{B}(\mathfrak{H})$ is amenable,
weakly amenable or character amenable if and only if $T$ is similar
to a normal operator. The purpose of this paper is to demonstrate an
example of compact triangular operator on an infinite-dimensional
Hilbert space which is weakly amenable and character amenable but is
not similar to a normal operator.

\vskip1cm
\section{Compact triangular weakly amenable operator}

Suppose $\sigma$ is a compact Hausdorff space, Let $C(\sigma)$
denote the Banach algebra of all continuous functions on $\sigma$
with the supremum norm $||f||_\infty=\sup_{x\in\sigma}|f(x)|$.
Throughout this paper we  let $\sigma=\{0,\lambda_1,
\lambda_2,\cdots\}$, where $\{\lambda_n\}_{n=1}^\infty$ is a
sequence of positive real numbers which converge to zero. Let
$$T=
\left(\begin{array}{cc}
0&N^{\frac{1}{2}}\\
0&N\\
\end{array}\right),$$ where $N$ is a normal operator with spectrum
$\sigma$, then $T$ is an operator on an infinite-dimensional Hilbert
space. In this section, we obtain that $T$ is weakly amenable, but
is not similar to a normal operator. Especially, if let
$$N= \left(\begin{array}{cccc}
\lambda_1& & &\\
         &\lambda_2&&\\
         &&\lambda_3&\\
         &         &&\ddots\\
\end{array}\right),$$
then $T$ is a compact triangular operator.

The following lemma is easily verified:
\begin{lem}\label{lem 1}
Suppose $\mathfrak{A}$ is a commutative Banach algebra which is
generated by the idempotent elements in $\mathfrak{A}$, then
$\mathfrak{A}$ is weak amenable.
\end{lem}

\begin{proof}
Let $\mathcal{P}$ denote the sets of the idempotent elements in
$\mathfrak{A}$.  Assume $X$ commutative Banach
$\mathfrak{A}$-bimodules, and $D : \mathfrak{A}\rightarrow X$ is a
derivation. For any $p\in \mathcal{P}$, $D(p)=D(p^2)=D(p^3)$ and
$D(p^2)=2pD(p), D(p^3)=3p^2D(p)$, so $D(p)=0$. Since
$p\in\mathcal{P}$ is  arbitrary and  $\mathfrak{A}$ is generated by
$\mathcal{P}$, it follows that $D(a)=0$ for all $a\in\mathfrak{A}$.
That is to say, $\mathfrak{A}$ is weak amenable.
\end{proof}

Our main result in this section will be that for any normal operator
$N$ with spectrum $\sigma$, $T= \left(\begin{array}{cc}
0&N^{\frac{1}{2}}\\
0&N\\
\end{array}\right)$
is weakly amenable but is not similar to a
normal operator.

\begin{thm}\label{thm1}
 Let $T= \left(\begin{array}{cc}
0&N^{\frac{1}{2}}\\
0&N\\
\end{array}\right),$
where $N$ is a normal operator with spectrum $\sigma$, then $T$ is
weakly amenable.
\end{thm}

\begin{proof}
By Lemma \ref{lem 1}, it suffices to show that $\mathfrak{A}_T$ is
generated by the idempotent elements in it.

Step 1. $\mathfrak{A}_T=\{\left(\begin{array}{cc}
0&f(N)\\
0&N^{\frac{1}{2}}f(N)\\
\end{array}\right); f\in C(\sigma), f(0)=0\}\triangleq M$, where
$f(N)$ denotes the functional calculus for $N$ respective to $f$.

Indeed,  for any polynomial $p(z)=\Sigma_{k=1}^na_kz^k=z\Sigma_{k=0}^{n-1}a_{k+1}z^k\triangleq zq(z)$, $p(T)$
has the form
 $$\left(\begin{array}{cc}
0&N^{\frac{1}{2}}q(N)\\
0&p(N)\\
\end{array}\right).$$
For any $A=\left(\begin{array}{cc}
0&A_{12}\\
0&A_{22}\\
\end{array}\right)
\in \mathfrak{A}_T$, there exists a sequence of polynomials
$\{p_n\}, p_n(0)=0$ for all $n$ such that $||p_n(T)-A||\rightarrow
0$. i.e. $||p_n(N)-A_{22}||\rightarrow 0$ and
$||N^{\frac{1}{2}}q_n(N)-A_{12}||\rightarrow 0$. Therefore, there
exists a function $g$ on $\sigma$, such that
$||z^\frac{1}{2}q_n-g||_\infty\rightarrow 0$
 and $||p_n-z^\frac{1}{2}g||_\infty\rightarrow 0$. It follows that
$A=\left(\begin{array}{cc}
0&g(N)\\
0&N^{\frac{1}{2}}g(N)\\
\end{array}\right),$
and $g\in C(\sigma), g(0)=0$. That is to say,
$\mathfrak{A}_T\subseteq M$.

For any $f\in C(\sigma), f(0)=0$, there exists a sequence of
polynomials $\{p_n, p_n(0)=0\}$ such that
$||p_n-f||_\infty\rightarrow 0$. Let $p_n=zq_n$, for any $n$ there
exists a polynomial $r_n, r_n(0)=0$ such that
$||r_n-z^\frac{1}{2}q_n||_\infty<\frac{1}{n}$. Therefore,
$||z^\frac{1}{2}r_n-f||_\infty\leq||z^\frac{1}{2}(r_n-z^\frac{1}{2}q_n)||_\infty+||p_n-f||_\infty\rightarrow
0$, and $||zr_n-z^\frac{1}{2}f||_\infty\rightarrow 0$. It follows
that $Tr_n(T)\rightarrow\left(\begin{array}{cc}
0&f(N)\\
0&N^{\frac{1}{2}}f(N)\\
\end{array}\right)$. That is to say, $M\subseteq\mathfrak{A}_T$.

Step 2. $\mathfrak{A}_T$ is generated by the idempotent elements in
it.

It is verity that for any $\lambda_n$ let
\begin{equation*}
h_n(z)=\begin{cases} \frac{1}{\sqrt{\lambda_n}} &
z=\lambda_n;\\
0, & z\neq\lambda_n,
\end{cases}
\end{equation*}
then $\left(\begin{array}{cc}
0&h_n(N)\\
0&N^{\frac{1}{2}}h_n(N)\\
\end{array}\right)$
 is an idempotent element in $\mathfrak{A}_T$, and
  $\mathfrak{A}_T$ is generated by  idempotent elements
  $\{\left(\begin{array}{cc}
0&h_n(N)\\
0&N^{\frac{1}{2}}h_n(N)\\
\end{array}\right)\}_{n=1}^\infty$. The proof is completed.
\end{proof}

Finally, we will obtain that $T$ is not similar to a normal operator. Indeed, Suppose $T$ is similar to a normal
operator, by the proof of \cite{G2006} Theorem 2.1 and \cite{F2005} Theorem 2.7, $T$ is amenable and there
exists an bounded operator $B$ such that
$$\left(\begin{array}{cc}
I&B\\
0&I\\
\end{array}\right)
\left(\begin{array}{cc}
0&N^\frac{1}{2}\\
0&N\\
\end{array}\right)\left(\begin{array}{cc}
I&-B\\
0&I\\
\end{array}\right)=\left(\begin{array}{cc}
0&0\\
0&N\\
\end{array}\right).$$ Therefore, there exists an bounded operator $B$
such that $BN=N^\frac{1}{2}$, which is impossible. Hence, $T$ is not
similar to a normal operator.

\begin{rem}
Theorem \ref{thm1} shows that there exists a compact triangular
operator with infinite spectrum which is a weakly amenable operator
but is not similar to a normal operator. However, we do not know if
a compact quasinilpotent operator can  be weakly amenable. It would
be very interesting to know whether this result is true. Indeed, if
any compact quasinilpotent operator can not be weakly amenable, then
it is easy to get that a compact operator $T$ with finite spectrum
is weakly amenable if and only if $T$ is similar to a normal
operator.
\end{rem}

\vskip1cm
\section{Compact triangular character amenable operator}
Let $Q$ be the Volterra operator on infinite-dimension Hilbert space, then by \cite{M} Corollary 2.7 and
\cite{D1988} Corollary 5.11, $Q$ is a compact quasinilpotent operator which is character amenable. However, the
lattice of invariant subspaces of $Q$ is a continuous nest. i.e. $Q$ is not a triangular operator.
 In this section, we will prove
that $T= \left(\begin{array}{cc}
0&N^{\frac{1}{2}}\\
0&N\\
\end{array}\right)$ is a character amenable operator, for any normal operator $N$ with
spectrum $\sigma$. Hence, there exists a compact triangular operator
which is a character amenable operator but is not similar to a
normal operator.

In \cite{M}, Sangani-Monfared obtained a necessary and sufficient
condition for a Banach algebra to be character amenable:
\begin{lem}\label{lem2}
A Banach algebra $\mathfrak{A}$ is character amenable if and only if $\ker\phi$ has a bounded approximate
identity for every $\phi\in \sigma(\mathfrak{A})\cup \{0\}$.
\end{lem}

Using Lemma \ref{lem2}, we will prove $T$ is a character amenable operator:

\begin{thm}\label{thm2}
 Let $T= \left(\begin{array}{cc}
0&N^{\frac{1}{2}}\\
0&N\\
\end{array}\right),$ then $T$ is character amenable.
\end{thm}

\begin{proof}
By Lemma \ref{lem2}, it suffices to show that $\mathfrak{A}_{T}$ and
$\mathfrak{A}_{\lambda_nI-T}$ have a bounded approximate identity
for all $n$.

It is verity that $\mathfrak{A}_{T}$ has  a bounded approximate
identity, hence we only to need prove that
$\mathfrak{A}_{\lambda_nI-T}$ has a bounded approximate identity for
any $n$.

For some fix $n$, assume that $f_n$ is a smooth function defined on
$[\lambda_n-\lambda_1, \lambda_n]$ and satisfies that
\begin{equation*}
f_n(z)=\begin{cases} 0 &
z=0;\\
1, & z\in[\lambda_n-\lambda_1,
\lambda_n-\lambda_{n-1}]\cup[\lambda_n-\lambda_{n+1}, \lambda_n].
\end{cases}
\end{equation*}
There exists a sequence of polynomials $\{p_k, p_k(0)=0\}$ such that
$p_k$ and $p_k'$ (the derivative of $p_k$ ) converge to $f_n$ and
$f_n'$ uniformly on $[\lambda_n-\lambda_1, \lambda_n]$, respectively
. Since $\lambda_nI-N$ is a normal operator with spectrum
$\sigma(\lambda_nI-N)=\{\lambda_n, \lambda_n-\lambda_1,
\lambda_n-\lambda_2,\cdots\}$, it follows that
 $f_n(\lambda_nI-N)$ is an identity of $\mathfrak{A}_{\lambda_nI-N}$ and $f_n(\lambda_nI)=\lambda_nI$.
 Hence $||(\lambda_nI-N)p_k(\lambda_nI-N)-(\lambda_nI-N)||\rightarrow 0$
 and $||p_k(\lambda_nI)-I||\rightarrow 0$, when $k\longrightarrow
 \infty$.

Note that if $T$ has the form
$$T= \left(\begin{array}{cc}
N_1&N_2\\
0&N_3\\
\end{array}\right),$$
with  $\{N_i\}$ a collection of commuting operators, then
$$T^k= \left(\begin{array}{cc}
N_1^k&A_k N_2\\
0&N_3^k\\
\end{array}\right),$$
where $N_1^k-N_3^k=(N_1-N_3)A_k$ for all $k\in\mathbb{N}$.

 It is easy to check that
$$(\lambda_nI-T)p_k(\lambda_nI-T)= \left(\begin{array}{cc}
\lambda_np_k(\lambda_nI)&-q_k(N)N^{\frac{1}{2}}\\
0&(\lambda_nI-N)p_k(\lambda_nI-N)\\
\end{array}\right),$$
where $\{q_k\}$ is a sequence of polynomials which satisfy the
equation
$\lambda_np_k(\lambda_n)-(\lambda_n-z)p_k(\lambda_n-z)=zq_k(z)$.
Note that
$q_k(z)-p_k(\lambda_n-z)=\frac{\lambda_np_k(\lambda_n)-\lambda_np_k(\lambda_n-z)}{z}=\lambda_np_k'(\xi_{k,n,z})$
for some $\xi_{k,n,z}\in[\lambda_n-\lambda_1, \lambda_n]$ and for
all $z\in\sigma$. Hence $\{q_k\}$ is bounded on $\sigma$.

 Since
$||Nq_k(N)-N||=||\lambda_np_k(\lambda_nI)-(\lambda_nI-N)p_k(\lambda_nI-N)-N||\rightarrow
0$, it follows that $\{q_k(N)\}$ is a bounded approximate identity
for $\mathfrak{A}_{N}$. Hence
$||N^{\frac{1}{2}}q_k(N)-N^{\frac{1}{2}}||\rightarrow 0$. Therefore
$\{p_k(\lambda_nI-T)\}$ is a bounded approximate identity for
$\mathfrak{A}_{\lambda_nI-T}$.
 \end{proof}

\begin{rem}
Theorem \ref{thm2} shows that there exists a compact triangular
operator with infinite spectrum which is a character amenable
operator but is not similar to a normal operator.  Moreover, by
Lemma \ref{lem2} we can describe  character amenable operator with
finite spectrum: If $T\in\mathfrak{B}(\mathfrak{H})$ with finite
spectrum $\sigma(T)=\{\delta_1, \delta_2, \cdots\delta_n\}$, then
$T$ is similar to
$$\left(\begin{array}{cccc}
\delta_1I+Q_1&&&\\
&\delta_2I+Q_2&&\\
&&\ddots&\\
&&&\delta_nI+Q_n\\
\end{array}\right),$$
where $Q_k$ is a quasinilpotent operator for $1\leq k\leq n$. By Lemma \ref{lem2}, $T$ is character amenable if
and only if $\mathfrak{A}_{Q_k}$ has a bounded approximate identity for $1\leq k\leq n$.
\end{rem}

\nocite{liyk2,liyk/kua1}

\begin{thebibliography}{4}

\bibitem{W1987}
W. G. Bade, P. C. Curtis and H. G. Dales, Amenability and weak
amenability for Beurling and Lipschitz algebras. Proc. London Math.
Soc. 55 (1987) 359--377.

\bibitem{D1988}
K. R. Davidson, Nest algebras, Longman group UK limited, Essex, 1988.


\bibitem{H2000}
H. G. Dales, Banach Algebras and Automatic Continuity (Oxford,
2000).







\bibitem{F2005} D. R. Farenick, B.E. Forrest and L. W. Marcoux,
Amenable operators on Hilbert spaces, J. reine angew. Math. 582
(2005) 201-228.

\bibitem{F2007} D. R. Farenick, B.E. Forrest and L. W. Marcoux,
Amenable operators on Hilbert spaces, J. reine angew. Math. 602
(2007) 235.


\bibitem{G2006}
J. A. Gifford, Operator algebras with a reduction property, J. Aust.
Math. Soc. 80 (2006) 297-315.



\bibitem{BE1972}
B. E. Johnson. Cohomology in Banach Algebras. Mem. Amer. Math. Soc.
Vol. 127 (Amer. Math. Soc., 1972).

\bibitem{M}
Monfared, Mehdi Sangani, Character amenability of Banach algebras,
Math. Proc. Cambridge Philos. Soc. 144 (2008) 697--706.



\bibitem{W1995}
G. A. Willis, When the algebra generated by an operator is amenable,
J. Operator Theorey. 34 (1995) 239--249.





\end{thebibliography}

\end{document}